\documentclass[sn-mathphys]{sn-jnl}

\jyear{2021}%

\theoremstyle{thmstyleone}%
\newtheorem{theorem}{Theorem}
\theoremstyle{thmstyletwo}%
\newtheorem{example}{Example}%
\newtheorem{remark}{Remark}%

\theoremstyle{thmstylethree}%
\newtheorem{definition}{Definition}%

\newtheorem{lemma}[theorem]{Lemma}
\theoremstyle{definition}

\numberwithin{equation}{section}

\begin{document}

\title[Levitin-Polyak Well-posedness for Split Equilibrium Problems]{Levitin-Polyak Well-posedness for Split Equilibrium Problems}


\author*[1]{\fnm{Soumitra} \sur{Dey}}\email{deysoumitra2012@gmail.com}

\author*[2]{\fnm{Aviv} \sur{Gibali}}\email{avivg@braude.ac.il}

\author*[1]{\fnm{Simeon} \sur{Reich}}\email{sreich@technion.ac.il}

\affil*[1]{\orgdiv{Department of Mathematics}, \orgname{The Technion -- Israel Institute of Technology}, \orgaddress{ \city{Haifa}, \postcode{3200003},  \country{Israel}}}

\affil*[2]{\orgdiv{Department of Mathematics}, \orgname{Braude College}, \orgaddress{\city{Karmiel}, \postcode{2161002}, \country{Israel}}}


\abstract{The notion of well-posedness has drawn the attention of many researchers in the field of nonlinear analysis, as it allows to explore problems in which exact solutions are not known and/or computationally hard to compute. Roughly speaking, for a given problem, well-posedness guarantees the convergence of approximations to exact solutions via an iterative method.
Thus, in this paper we extend the concept of Levitin-Polyak well-posedness to split equilibrium problems in real Banach spaces. In particular, we establish a metric characterization of Levitin-Polyak well-posedness by perturbations and also show an equivalence between Levitin-Polyak well-posedness by perturbations for split equilibrium problems and the existence and uniqueness of their solutions.}




\maketitle

\section{Introduction}
\label{Sec:1}
The concept of well-posedness is one of the most important and interesting topics in nonlinear analysis and optimization theory. It has been studied in connection with many problems, such as variational inequalities (VIs), equilibrium problems (EPs) and inverse problems (IPs). The notion of well-posedness for a minimization problem was first introduced and examined by Tykhonov \cite{ANTY1966} in metric spaces. Recall that a minimization problem is said to be well-posed (or correctly posed) if its solution set is a singleton and every minimizing sequence converges to this unique solution (see, for instance, \cite{MFUR1970}). In practice, the minimizing sequence generated by a numerical method may fail to lie inside the feasible set although the sequence of the distances of its elements to the feasible set does tend to zero. Such a sequence is called a generalized minimizing sequence for the minimization problem at hand. Motivated and inspired by this fact, Levitin and Polyak \cite{ESLE1966} modified the notion of Tykhonov well-posedness by replacing minimizing sequences by generalized minimizing sequence. Since the requirement that the solution set be a singleton may be too restrictive, Furi and Vignoli \cite{MFUR1970} relaxed the above definition of well-posedness when applied to optimal control and mathematical programming problems. This encouraged researchers to study the notion of well-posedness of problems with multiple solutions. Motivated by this development, Zolezzi \cite{TZOL1995,TZOL1996} introduced the concept of parametric minimization problems and extended well-posedness and well-posedness by perturbations to a more general setting. Further developments in the area of well-posedness and Levitin-Polyak (LP) well-posedness can be found, for example, in \cite{XXHU2006,ALDO1993} and the references therein.

It is well known that minimization problems and variational inequality (VI) problems are closely related. Moreover, VIs are a useful tool for solving diverse analysis and optimization problems such as systems of linear and nonlinear equations, complementarity problems and saddle point problems. For more applications see, for instance,   \cite{FFAC2003,RFER2007} and references therein.

Let $E$ be a real Banach space with norm $\|\cdot\|$. Let $E^*$ be the dual of $E$, and let $\left\langle\cdot, \cdot\right\rangle$ be the dual pairing of $E^*$ and $E$. Let $C$ be a nonempty, closed and convex subset of $E$ and $g:E\rightarrow E^*$ be a single-valued mapping. Then a VI problem (see \cite{LOJO2021}) consists of finding a point $x^*\in C$ with the following property:
\begin{align}\label{VI}
\left\langle g(x^*), x-x^*\right\rangle\geq 0 \; \quad\forall x\in C.
\end{align}

Lucchetti and Patrone in \cite{RLUC1981,RLUC1982} were the first to extend the well-posedness concept to VIs. Thereafter, the Levitin-Polyak well-posedness was studied by Hu et al. \cite{RHU2010} as well as by Huang et al. \cite{XXHU2007,XXHU2009}. Later on, the concept of well-posedness by perturbations for mixed variational inequality problems was introduced by Feng et al. \cite{YPFA2010}. Various types of well-posedness for different classes of hemivariational inequality problems were studied by many authors (see, for example, \cite{DGOE1995,YBXI2011,JCEN2022}). Numerical algorithms for solving VIs can be found in \cite{FFAC2003}.
At this point we recall that VIs were first introduced and studied by Stampacchia \cite{GSTA1964} (in finite-dimensional Euclidean spaces). Thereafter many researchers studied VIs and explored various directions in both finite- and infinite-dimensional settings. See, for instance, \cite{FFAC2003, YCEN2010,QLDO2017,TOAL2021}.

Equilibrium problems (EP) constitute an important generalization of variational inequalities. Let $C$ be a nonempty, closed and convex subset of a real Banach space $E$, and let $f:E\times E\rightarrow\mathbb{R}$ be a bifunction. With these data, the equilibrium problem (see \cite{PSUN2013}), denoted by EP$(f,C)$, consists of finding a point $x^*\in C$ such that
\begin{align}\label{EP}
f(x^*, x)\geq 0 \; \quad\forall  x\in C.
\end{align}

Equilibrium problems were first introduced and studied by Blum and Oettli \cite{EBLU1994} on a real topological vector space. In recent years EPs have received a lot of attention because they contain as special cases diverse problems such as optimization problems \cite{TZOL1996}, saddle point problems \cite{MBIA2010}, VIs \cite{GSTA1964,EBLU1994} and Nash EPs \cite{FFAC2003}. Moreover, EPs stand at the core of many applications as a modelling tool for various practical problems arising {\it inter alia} in game theory, physics and economics.
Another important feature of EPs is that they make it possible to consider all these special cases in a unified manner. As EPs are a common extension of minimization problems, saddle point problems and VIs, their well-posedness concepts originate in the well-posedness concepts already introduced for these problems (see, for instance, \cite{YPFA22008,MBIA2010} and references therein). Iterative algorithms for solving EPs can be found, for example, in \cite{DVHI2020,YLIU2019,LOJO2020,DVHI2018,HURE2021,PNAN2012,POLI2013}. More recently, Cotrina et al. \cite{JCOT2019} studied EPs in Euclidean spaces and introduced the concept of projected solutions to EPs with moving constraint sets. These are commonly known as quasi-equilibrium problem (QEP).

Censor et al. \cite[Section 2]{YCEN2010} introduced the general split inverse problem (SIP) in which there are given two vector spaces,
$X$ and $Y$, and a bounded linear operator $A:X\rightarrow Y$. In addition, two inverse problems are involved. The first one, denoted by IP$_{1}$, is
formulated in the space $X$ and the second one, denoted by IP$_{2}$, is formulated in the space $Y$. Given these data, the Split Inverse Problem
(SIP) is formulated as follows:%
\begin{gather}
\text{find a point }x^{\ast }\in X\text{ that solves IP}_{1} \\
\text{such that}  \notag \\
\text{the point }y^{\ast }=Ax^{\ast }\in Y\text{ solves IP}_{2}.
\end{gather}%

SIPs are quite general because they enable one to obtain various split problems by making different choices of IP$_{1}$ and IP$_{2}$. An important example is the split variational inequality (SVI). Let $E_1$ and $E_2$ be two real Banach spaces with duals $E_1^*$ and $E_2^*$, respectively, and let $C$ and $Q$ be two nonempty, closed and convex subsets of $E_1$ and $E_2$, respectively. Let $f:E_1\rightarrow E_1^*$ and $g:E_2\rightarrow E_2^*$ be two single-valued mappings and let $A:E_1\rightarrow E_2$ be a bounded linear operator. The SVI is defined as follows (see \cite{JLI2018}):
{\small \begin{equation}\label{SVIP}
\text{Find } (x^*,y^*)\in E_1\times E_2\text{ such that}
\begin{cases}
x^*\in C,\quad y^*\in Q,\quad y^*=Ax^*,\\
\left\langle f(x^*), x-x^*\right\rangle \geq 0 \quad\forall x\in C,\\
\left\langle g(y^*), y-y^*\right\rangle\geq 0 \quad\forall y\in Q.
\end{cases}
\end{equation}}

Recently, Hu et al. \cite{RHUA2016} introduced and studied the LP well-posedness of the SVI (\ref{SVIP}) in real reflexive Banach spaces. For more details regarding SVIs, the reader is referred to \cite{YCEN2010,BALI2019,QHAN2014} and references therein. Observe that if $C=Q$, $f=g$, $E_1=E_2=E$ and $A=I$, where $I$ denotes the identity operator, then the SVI (\ref{SVIP}) is reduced to the classical VI (\ref{VI}). Also, it is clear that if $f=g=0$, then the SVI is reduced to the very well-known split feasibility problem (SFP), which was first introduced and studied by Censor and Elfving \cite{YCEN1994} in finite-dimensional Euclidean spaces. Also, an important class of SIP was studied by Moudafi \cite{AMOU2011} in infinite-dimensional Hilbert spaces.

Another SIP instance is the split equilibrium problem (SEP). Let $E_1$ and $E_2$ be two real Banach spaces, and let $C$ and $Q$ be two nonempty, closed and convex subsets of $E_1$ and $E_2$, respectively. Let $f:E_1\times E_1\rightarrow\mathbb{R}$ and $g:E_2\times E_2\rightarrow\mathbb{R}$ be two bifunctions, and let $A:E_1\rightarrow E_2$ be a bounded linear operator. With these data, the SEP$(f,g,C,Q)$ is defined as follows:
{\small \begin{equation}\label{SEP}
\text{Find } (x^*,y^*)\in E_1\times E_2\text{ such that}
\begin{cases}
x^*\in C,\quad y^*\in Q,\quad y^*=Ax^*,\\
 f(x^*, x)\geq 0 \quad\forall x\in C,\\
 g(y^*, y)\geq 0 \quad\forall y\in Q.
\end{cases}
\end{equation}}

In 2012 He \cite{ZHE2012} studied the SEP$(f,g,C,Q)$ and established weak and strong convergence theorems for a new class of iterative methods for solving it in real Banach spaces. Later, Dinh et al. \cite{BVDI2015} proposed new extragradient algorithms for solving the SEP$(f,g,C,Q)$ in infinite-dimensional Hilbert spaces. For more details and results in this direction, consult \cite{JLI2019,DVHI20202,YISU2021,MALA2020,DSKI2018}.

Recalling that perturbing a problem may help achieve stability \cite{TZOL1996}, and motivated and inspired by the concept of LP well-posedness by perturbations for variational inequalities, in the present paper we introduce and extend this concept to split equilibrium problems (\ref{SEP}).

Our paper is organized as follows. In Section \ref{Sec:2} some useful definitions and results are collected. In Section \ref{Sec:3} we introduce the concepts of an approximating sequence and a generalized approximating sequence for the SEP$(f,g,C,Q)$. Also in this section, we introduce various kinds of well-posedness for the SEP$(f,g,C,Q)$. In Section \ref{Sec:4} we establish the metric characterization of (generalized) LP well-posedness by perturbations for the SEP$(f,g,C,Q)$. In this section we have also provided some nontrivial examples to illustrate our theoretical analysis. In Section \ref{Sec:5} we prove that the well-posedness by perturbations of the SEP$(f,g,C,Q)$ is equivalent to the existence and uniqueness of its solution. Finally, in Section \ref{Sec:6} we present some conclusions.

\section{Preliminaries}
\label{Sec:2}
In this section we present some basic definitions and results which are useful in the rest of the paper. For more information and results, see, for example, \cite{EBLU1994,KKUR1968,EKLE1984,HHBA2011,VRAK1998,JCOT2008} and references therein.

Let $A$ and $B$ be two nonempty subsets of a real Banach space $E$. The {\em Hausdorff metric} $H(\cdot, \cdot)$ between $A$ and $B$ is defined by
\begin{align*}
H(A,B) := \max\left\lbrace D(A,B), D(B,A)\right\rbrace,
\end{align*}
where $D(A,B) := \sup_{a\in A}d(a,B)$ with $d(a,B) := \inf_{b\in B}\|a-b\|$.
Let $\left\lbrace A_n\right\rbrace$ be a sequence of subsets of $E$. We say that the sequence $\left\lbrace A_n\right\rbrace$ converges to $A$ if $H(A_n,A)\rightarrow 0$.
It is not difficult to see that $\{D(A_n,A)\}$ $\rightarrow 0$ if and only if $\{d(a_n,A)\} \rightarrow 0$, uniformly for all selections $a_n\in A_n$.

\begin{lemma}\label{Lem:2}
Let $C$ be a nonempty convex subset of a real Banach space $E$. Then for each $x\in E$, the distance function $d(x,C)$ (defined above) is continuous and convex.
\end{lemma}


The {\em diameter of a set} $A$ is defined by
\begin{align*}
\text{diam}(A) := \sup \left\lbrace \|x-y\|: x, y\in A\right\rbrace.
\end{align*}

The {\em Kuratowski measure of noncompactness} of a set $A$ is defined by
\begin{align}
\mu(A) := \inf\left\lbrace\epsilon>0: A\subset\cup_{i=1}^n A_i,  \text{diam}(A_i)<\epsilon,~ i=1,2,\cdots, n,~
n\in\mathbb{N} \right\rbrace\nonumber.
\end{align}

\begin{lemma}
Let $A$ and $B$ be two nonempty subsets of a real Banach space $E$.  Then
\begin{align*}
\mu(A)\leq 2 H(A,B)+\mu(B).
\end{align*}
\end{lemma}

\begin{definition}
Let $C$ and $Q$ be two nonempty subsets of real Banach spaces $E_1$ and $E_2$, respectively. A function $f:C\times Q\rightarrow\mathbb{R}$ is said to be {\em upper semi-continuous at a point} $(x,y)$ in $C\times Q$ if for every sequence $\{(x_n,y_n)\}\subset C\times Q$ with $(x_n,y_n)\rightarrow (x,y)$, we have
\begin{align*}
\limsup_{n\rightarrow\infty} f(x_n,y_n)\leq f(x,y).
\end{align*}
\end{definition}

\begin{definition}
Let $C$ be a nonempty subset of a real Banach space $E$. A bifunction $f:C\times C\rightarrow\mathbb{R}$ is said to be {\em monotone} if
\begin{equation}
f(x, y)+f(y, x)\leq 0 \; \quad\forall x, y\in C.\nonumber
\end{equation}
\end{definition}

\begin{definition}
Let $C$ be a nonempty convex subset of a real Banach space $E$. A bifunction $f:C\times C\rightarrow\mathbb{R}$ is said to be {\em hemicontinuous} if
for all $x,y\in C$ the function $t\in [0,1]\rightarrow f(x+t(y-x), y)$ is upper semicontinuous at $t=0$. That is, if
\begin{equation}
\limsup_{t\rightarrow 0^+}f(x+t(y-x), y)\leq f(x, y) \; \quad\forall x, y\in C.\nonumber
\end{equation}
\end{definition}

\begin{lemma}\label{Lem:1}
Let $C$ be convex subset of a real Banach space $E$ and let $f:C\times C\rightarrow\mathbb{R}$ be a monotone and hemicontinuous bifunction. Assume that
\begin{enumerate}
  \item $f(x, x)\geq 0$ for all $x\in C$.
  \item For every $x\in C$, $f(x,\cdot)$ is convex.
\end{enumerate}
Then for a given $x^*\in C,$ we have
\begin{equation*}
f(x^*, x)\geq 0 \; \quad \forall x\in C
\end{equation*}
if and only if
\begin{equation*}
f(x, x^*)\leq 0 \; \quad \forall x\in C.
\end{equation*}
\end{lemma}

\section{Approximating sequences and Levitin-Polyak well-posedness}
\label{Sec:3}

In this section we introduce approximating sequences and generalized approximating sequences for the SEP$(f,g,C,Q)$, and introduce the generalized Levitin-Polyak well-posedness notions for the SEP$(f,g,C,Q)$. For simplicity, we denote by $S$ the solution set of SEP$(f,g,C,Q)$.

Let $N$ be a normed space of parameters and let $M\subset N$ be a closed ball centered at $p^*\in M$ with positive radius.
The perturbed problem of SEP$(f,g,C,Q)$ is defined as follows: Find $(x^*, y^*)\in E_1\times E_2$ such that
{\small \begin{equation}\label{PSEP}
SEP_p(f,g,C,Q) \begin{cases}
x^*\in C,\quad y^*\in Q,\quad y^*=Ax^*,\\
\tilde{f}(p,x^*,x)\geq 0\quad\forall x\in C,\\
\tilde{g}(p,y^*,y)\geq 0\quad\forall y\in Q,
\end{cases}
\end{equation}}
where $\tilde{f}:M\times E_1\times E_1\rightarrow\mathbb{R}$ and $\tilde{g}:M\times E_2\times E_2\rightarrow\mathbb{R}$ are two functions such that
\begin{align}
\tilde{f}(p^*,\cdot,\cdot)=f(\cdot, \cdot) \text{ and } \tilde{g}(p^*,\cdot,\cdot)=g(\cdot, \cdot).\nonumber
\end{align}

Here $SEP(f,g,C,Q)$ is the original problem, while $SEP_p(f,g,C,Q)$ is its perturbation model corresponding to the parameter $p\in M$. We are interested in the behavior of the approximate solutions of these optimization problems under small perturbations around $p^*$.

\begin{definition}
Let $\{p_n\}\subset M$ with $p_n\rightarrow p^*$ and let $\{(x_n, y_n)\}$ be a sequence in $E_1\times E_2$.

\begin{enumerate}
\item The sequence $\{(x_n, y_n)\}$ is called an {\em approximating sequence} corresponding to $\{p_n\}$ for the SEP$(f,g,C,Q)$ if there exists a positive sequence $\{\epsilon_n\}\rightarrow 0$ such that
\begin{equation}\label{ASEP}
\begin{cases}
x_n\in C, \quad y_n\in Q, \quad\|y_n-Ax_n\|\leq \epsilon_n,\\
\tilde{f}(p_n, x_n, x)\geq -\epsilon_n \quad\forall x\in C,\\
\tilde{g}(p_n, y_n, y)\geq -\epsilon_n \quad\forall y\in Q.
\end{cases}
\end{equation}

\item The sequence $\{(x_n, y_n)\}$ is called a {\em generalized approximating sequence} corresponding to $\{p_n\}$ for the SEP$(f,g,C,Q)$ if there exists a positive sequence $\{\epsilon_n\}\rightarrow 0$ such that
\begin{equation}\label{Gen.ASEP}
\begin{cases}
d(x_n,C)\leq\epsilon_n, \quad d(y_n,Q)\leq\epsilon_n,\\
\|y_n-Ax_n\|\leq \epsilon_n,\\
\tilde{f}(p_n, x_n, x)\geq -\epsilon_n \quad\forall x\in C,\\
\tilde{g}(p_n, y_n, y)\geq -\epsilon_n \quad\forall y\in Q.
\end{cases}
\end{equation}
\end{enumerate}
\end{definition}

Observe that if
$p_n=p^*$ for every $n\in\mathbb{N}$ and $E_1$, $E_2$ are Hilbert spaces, then (\ref{ASEP}) coincides with the approximating sequence definition in \cite{SDEY2022}.

\begin{definition}\label{Def.:W&G wellposedness}
The SEP$(f,g,C,Q)$ is said to be well-posed by perturbations if its solution set $S$ is a singleton and for any $\{p_n\}\subset M$ with $p_n\rightarrow p^*$, every approximating sequence corresponding to $\{p_n\}$ for SEP$(f,g,C,Q)$ converges to the unique solution to the problem. We say that SEP$(f,g,C,Q)$ is well-posed by perturbations in the generalized sense (generalized well-posed by perturbations) if $S\neq\emptyset$ and for any $\{p_n\}\subset M$ with $p_n\rightarrow p^*$, every approximating sequence corresponding to $\{p_n\}$ for SEP$(f,g,C,Q)$ has a subsequence which converges to some element of $S$.
\end{definition}

\begin{remark}
If $f=g, C=Q, E_1=E_2=E$, and $A=I$, then Definition \ref{Def.:W&G wellposedness} reduces to the definitions of well-posedness and generalized well-posedness by perturbations for EP($f, C$). 
\end{remark}

\begin{definition}\label{Def.:LP W&G wellposedness}
The SEP$(f,g,C,Q)$ is called LP well-posed by perturbations if $S$ is a singleton and for any $\{p_n\}\subset M$ with $p_n\rightarrow p^*$, every generalized approximating sequence corresponding to $\{p_n\}$ for SEP$(f,g,C,Q)$ converges to the unique solution. We say that SEP$(f,g,C,Q)$ is LP well-posed by perturbations in the generalized sense (generalized LP well-posed by perturbations) if $S\neq\emptyset$ and for any $\{p_n\}\subset M$ with $p_n\rightarrow p^*$, every generalized approximating sequence corresponding to $\{p_n\}$ for SEP$(f,g,C,Q)$ has a subsequence which converges to some element of $S$.
\end{definition}

\begin{remark}
If $f=g, C=Q, E_1=E_2=E$ and $A=I$, then Definition \ref{Def.:LP W&G wellposedness} reduces to the definitions of LP well-posedness and generalized LP well-posedness by perturbations for $EP(f,C)$.
\end{remark}

\noindent Given $\epsilon \ge 0$, define the following set:
\begin{multline}
S(\epsilon) := \bigcup_{p\in\mathbb{B}(p^*,\epsilon)} \Big\{\left(z, w)\in E_1\times E_2: d(z,C)\leq\epsilon,  d(w,Q)\leq\epsilon, \right.\nonumber\\ \|w-Az\|\leq \epsilon;
 \left. \tilde{f}(p, z, x)\geq -\epsilon\quad\forall x\in C;~ \tilde{g}(p,w,y)\geq -\epsilon\quad\forall y\in Q\right\},
\end{multline}
where $\mathbb{B}(p^*,\epsilon)$ denotes the closed ball of radius $\epsilon$ centered at $p^*$.

The set $S(\epsilon)$ is called the approximate solution set of SEP$(f,g,C,Q)$. It is obvious that $S\subset S(\epsilon)$ for each $\epsilon>0$ and moreover,
$S(\epsilon)\subset S(\tilde{\epsilon})$ for every $0\leq\epsilon<\tilde{\epsilon}$.

\section{Metric characterization}
\label{Sec:4}
The metric characterization introduced by Furi and Vignoli \cite{MFUR1970} plays an essential role in the theory of well-posedness. In this section we establish a Furi-Vignoli type metric characterization of (generalized) LP well-posedness by perturbations for SEP$(f,g,C,Q)$. In this section of the paper we assume that $C$ and $Q$ are nonempty, closed and convex subsets of the real Banach spaces $E_1$ and $E_2$, respectively.

\begin{theorem}\label{Thm:1}
SEP$(f,g,C,Q)$ is LP well-posed by perturbations if and only if its solution set $S$ is nonempty and $\text{diam}(S(\epsilon))\rightarrow 0$ as $\epsilon\rightarrow 0$.
\end{theorem}

\begin{proof}
Suppose SEP$(f,g,C,Q)$ is LP well-posed by perturbations. By the definition of LP well-posedness by perturbations for SEP$(f,g,C,Q)$, the solution set $S$ is a singleton
and therefore nonempty.

We claim that $\text{diam}(S(\epsilon))\rightarrow 0$ as $\epsilon\rightarrow 0$.
Suppose to the contrary that $\text{diam}(S(\epsilon))\nrightarrow 0$ as $\epsilon\rightarrow 0$.
Then there exist $\delta>0$, $0<\epsilon_n\rightarrow 0$, $(x_n, y_n)\in S(\epsilon_n)$, and $(\tilde{x}_n, \tilde{y}_n)\in S(\epsilon_n)$ such that
\begin{align}\label{Thm:con1}
\|(x_n, y_n)-(\tilde{x}_n, \tilde{y}_n)\|>\delta \; \quad\forall n\in\mathbb{N}.
\end{align}
Since $(x_n, y_n)\in S(\epsilon_n)$ and $(\tilde{x}_n, \tilde{y}_n)\in S(\epsilon_n)$, there exist $p_n\in\mathbb{B}(p^*,\epsilon_n)$ and $\tilde{p}_n\in\mathbb{B}(p^*,\epsilon_n)$
such that
\begin{equation*}\label{Gen.ASEP1}
\begin{cases}
d(x_n,C)\leq\epsilon_n, \quad d(y_n,Q)\leq\epsilon_n,\\
\|y_n-Ax_n\|\leq \epsilon_n,\\
\tilde{f}(p_n, x_n, x)\geq -\epsilon_n\quad\forall x\in C,\\
\tilde{g}(p_n, y_n, y)\geq -\epsilon_n\quad\forall y\in Q,
\end{cases}
\end{equation*}
and
\begin{equation*}\label{Gen.ASEP2}
\begin{cases}
d(\tilde{x}_n,C)\leq\epsilon_n, \quad d(\tilde{y}_n,Q)\leq\epsilon_n,\\
\|\tilde{y}_n-A\tilde{x}_n\|\leq \epsilon_n,\\
\tilde{f}(\tilde{p}_n, \tilde{x}_n, x)\geq -\epsilon_n\quad\forall x\in C,\\
\tilde{g}(\tilde{p}_n, \tilde{y}_n, y)\geq -\epsilon_n\quad\forall y\in Q.
\end{cases}
\end{equation*}
When combined with the facts that $p_n\rightarrow p^*$ and $\tilde{p}_n\rightarrow p^*$, this implies that $\{(x_n, y_n)\}$ (respectively, $\{(\tilde{x}_n, \tilde{y}_n)\}$) is a generalized approximating sequence corresponding to $\{p_n\}$ (respectively, $\{\tilde{p}_n\}$) for SEP$(f,g,C,Q)$.
Therefore, $\{(x_n, y_n)\}$ and $\{(\tilde{x}_n, \tilde{y}_n)\}$ both converge to the unique solution of SEP$(f,g,C,Q)$, a contradiction to (\ref{Thm:con1}).

Conversely, suppose $S$ is nonempty and diam$(S(\epsilon))\rightarrow 0$ as $\epsilon\rightarrow 0$.
It is clear that $S\subset S(\epsilon)$ for every $\epsilon>0$. Since diam$(S(\epsilon))\rightarrow 0$ as $\epsilon\rightarrow 0$ and $S\subset S(\epsilon)$ for every $\epsilon>0$, it follows that $S$ is a singleton.

Let $\{p_n\}\subset M$ with $p_n\rightarrow p^*$ and let $\{(x_n,y_n)\}$ be a generalized approximating sequence corresponding to $\{p_n\}$ for SEP$(f,g,C,Q)$. Then there exists $0<\epsilon_n\rightarrow 0$ such that
\begin{equation*}
\begin{cases}
d(x_n,C)\leq\epsilon_n, \quad d(y_n,Q)\leq\epsilon_n,\\
\|y_n-Ax_n\|\leq \epsilon_n,\\
\tilde{f}(p_n, x_n, x)\geq -\epsilon_n\quad\forall x\in C,\\
\tilde{g}(p_n, y_n, y)\geq -\epsilon_n\quad\forall y\in Q.
\end{cases}
\end{equation*}
Let $\tilde{\epsilon}_n=\max\{\epsilon_n, \|p_n-p^*\|\}$. Then $(x_n,y_n)\in S(\tilde{\epsilon}_n)$ and $\tilde{\epsilon}_n\rightarrow 0$ as $n\rightarrow\infty$.

Let $(x^*,y^*)$ be the unique solution of SEP$(f,g,C,Q)$. It is clear that $(x^*,y^*)\in S(\tilde{\epsilon}_n)$. Therefore,
\begin{align*}
\|(x_n,y_n)-(x^*,y^*)\|\leq\text{diam}(S(\tilde{\epsilon}_n))\rightarrow 0, \text{ as } n\rightarrow\infty.
\end{align*}
Since $\{(x_n,y_n)\}$ is arbitrary, this means that every generalized approximating sequence corresponding to $\{p_n\}$ converges to the unique solution of SEP$(f,g,C,Q)$.

Hence, SEP$(f,g,C,Q)$ is LP well-posed and the proof is complete.
\end{proof}

\begin{remark}
The above Theorem \ref{Thm:1} provides the equivalence between the LP well-posedness of $SEP(f,g,C,Q)$ by perturbations and properties of its solution and approximate solution sets.
\end{remark}

The following example illustrates our Theorem \ref{Thm:1}.

\begin{example}
Let $E_1=E_2=\mathbb{R}$, $C=[-1,0]=Q$, $M=[-1,1]\subset\mathbb{R}$, $p^*=0$, $A=I,$ the identity operator.

Define $\tilde{f}:M\times E_1\times E_1\rightarrow\mathbb{R}$ and $\tilde{g}:M\times E_2\times E_2\rightarrow\mathbb{R}$ by
\begin{align*}
&\tilde{f}(p,z,x)=(z-x)(p^2+2), \quad p\in M, x,z\in E_1~~\text{and}\\
&\tilde{g}(p,w,y)=w-y, \quad p\in M, w,y\in E_2,
\end{align*}
with
\begin{align*}
&\tilde{f}(0,z,x)=2(z-x)=f(z,x), \quad x,z\in E_1~~\text{and}\\
&\tilde{g}(0,w,y)=w-y=g(w,y), \quad w,y\in E_2.
\end{align*}
It is not difficult to check that $S=\{(0,0)\}$ is the unique solution of the $SEP(f,g,C,Q)$. In particular, the solution set $S$ is nonempty.

For any $\epsilon>0$, the approximate solution set of the $SEP(f,g,C,Q)$ is given by
\begin{multline}
S(\epsilon)= \bigcup_{p\in\mathbb{B}(0,\epsilon)} \Big\{\left(z, w)\in\mathbb{R}\times\mathbb{R}: d(z,[-1,0])\leq\epsilon,  d(w,[-1,0])\leq\epsilon, \right. \vert w-z\vert\leq \epsilon;\nonumber\\
 \left. (z-x)(p^2+2)\geq -\epsilon\quad\forall x\in [-1,0];~ w-y\geq -\epsilon\quad\forall y\in [-1,0]\right\}.
 \end{multline}

Now,
\begin{align*}
&~~~~(z-x)(p^2+2)\geq -\epsilon\quad\forall x\in [-1,0].\\
&\Rightarrow z-x\geq-\epsilon/(p^2+2)\quad\forall x\in [-1,0].\\
&\Rightarrow z\geq x-\epsilon/(p^2+2)\quad\forall x\in [-1,0].\\
&\Rightarrow z\geq-\epsilon/(p^2+2).
\end{align*}

Similarly,
\begin{align*}
&~~~~w-y\geq -\epsilon\quad\forall y\in [-1,0].\\
&\Rightarrow w\geq y-\epsilon\quad\forall y\in [-1,0].\\
&\Rightarrow w\geq-\epsilon.
\end{align*}

Also,
\begin{align*}
\vert z-w\vert\leq\epsilon,\quad d(z,[-1,0])\leq\epsilon,\quad d(w,[-1,0])\leq\epsilon.
\end{align*}

Combining all of the above, one can check that
\begin{align*}
S(\epsilon)&~\subset \bigcup_{p\in\mathbb{B}(0,\epsilon)} \Big\{[-\epsilon/(p^2+2),\epsilon]\times[-\epsilon,\epsilon]\Big\}\\&
= [-\epsilon/2,\epsilon]\times[-\epsilon,\epsilon]
\end{align*}
for sufficiently small $\epsilon>0$.
Therefore, $diam(S(\epsilon))\leq\frac{5\epsilon}{2}\rightarrow 0$ as $\epsilon\rightarrow 0$.

By the above Theorem \ref{Thm:1}, $SEP(f,g,C,Q)$ is LP well-posed by perturbations.
\end{example}

Our next result shows that if $\tilde{f}$ and $\tilde{g}$ are two upper semi-continuous functions, then Theorem \ref{Thm:1} still holds under a weaker assumption on the solution set $S$ of SEP$(f,g,C,Q)$.

\begin{theorem}\label{Thm:2}
Let $\tilde{f}:M\times E_1\times E_1\rightarrow\mathbb{R}$ and let $\tilde{g}: M\times E_2\times E_2\rightarrow\mathbb{R}$ be two functions such that $\tilde{f}(\cdot,\cdot,x)$ and $\tilde{g}(\cdot,\cdot,y)$ are upper semi-continuous for each $(x,y)\in E_1\times E_2$. Then SEP$(f,g,C,Q)$ is LP well-posed by perturbations if and only if
\begin{align}\label{Thm2:con1}
S(\epsilon)\neq\emptyset\quad\forall\epsilon>0\quad\text{and}\quad \text{diam}(S(\epsilon))\rightarrow 0 \text{ as } \epsilon\rightarrow 0.
\end{align}
\end{theorem}

\begin{proof}
The necessity part follows directly from Theorem \ref{Thm:1}. For the sufficiency part, suppose condition (\ref{Thm2:con1}) holds. Since $S\subset S(\epsilon)$ for each $\epsilon>0$, SEP$(f,g,C,Q)$ admits at most one solution.

Let $\{p_n\}\subset M$ with $p_n\rightarrow p^*$ and let $\{(x_n,y_n)\}$ be a generalized approximating sequence corresponding to $\{p_n\}$ for SEP$(f,g,C,Q)$.
Then there exists a positive sequence $\{\epsilon_n\}$ such that  $\epsilon_n \rightarrow 0$ and
\begin{equation*}
\begin{cases}
d(x_n,C)\leq\epsilon_n, \quad d(y_n,Q)\leq\epsilon_n,\\
\|y_n-Ax_n\|\leq \epsilon_n,\\
\tilde{f}(p_n, x_n, x)\geq -\epsilon_n\quad\forall x\in C,\\
\tilde{g}(p_n, y_n, y)\geq -\epsilon_n\quad\forall y\in Q.
\end{cases}
\end{equation*}
Set $\tilde{\epsilon}_n=\max\{\epsilon_n, \|p_n-p^*\|\}$. Then $(x_n,y_n)\in S(\tilde{\epsilon}_n)$ and $\tilde{\epsilon}_n\rightarrow 0$ as $n\rightarrow\infty$.
From condition (\ref{Thm2:con1}) it follows that $\{(x_n,y_n)\}$ is a Cauchy sequence. Let $(x_n,y_n)\rightarrow(x^*,y^*)$.
Since $\tilde{f}(\cdot,\cdot,x)$ and $\tilde{g}(\cdot,\cdot,y)$ are upper semi-continuous for each $(x,y)\in E_1\times E_2$, and $d(\cdot,C)$ is continuous, we have
\begin{equation*}
\begin{cases}
x^*\in C, \quad y^*\in Q, \quad y^*=Ax^*,\\
\tilde{f}(p^*, x^*, x)\geq 0\quad\forall x\in C,\\
\tilde{g}(p^*, y^*, y)\geq 0\quad\forall y\in Q.
\end{cases}
\end{equation*}
That is,
\begin{equation*}
\begin{cases}
x^*\in C, \quad y^*\in Q, \quad y^*=Ax^*,\\
f(x^*, x)\geq 0\quad\forall x\in C,\\
g(y^*, y)\geq 0\quad\forall y\in Q.
\end{cases}
\end{equation*}

This implies that $(x^*, y^*)$ is the unique solution to SEP$(f,g,C,Q)$.
Thus SEP$(f,g,C,Q)$ is indeed LP well-posed by perturbations, as asserted.
\end{proof}

\begin{remark}
The above Theorem \ref{Thm:2} not only provides the LP well-posedness of $SEP(f,g,C,Q)$ by perturbations, but also its equivalence with the existence and uniqueness of its solution.
\end{remark}

Next, we use the Hausdorff metric to present a characterization of generalized LP well-posedness by perturbations for the SEP$(f,g,C,Q)$. The concept of generalized well-posedness enables to remove the assumption that the solution set $S$ of the SEP$(f,g,C,Q)$ is a singleton.

\begin{theorem}\label{Thm:3}
SEP$(f,g,C,Q)$ is LP well-posed by perturbations in the generalized sense if and only if the solution set $S$ of SEP$(f,g,C,Q)$ is nonempty compact and
\begin{align}\label{Thm3:con1}
H(S(\epsilon),S)\rightarrow 0\text{ as }\epsilon\rightarrow 0.
\end{align}
\end{theorem}

\begin{proof}
First, we tackle the necessity part. To this end, suppose the SEP$(f,g,C,Q)$ is LP well-posed by perturbations in the generalized sense. It is clear that $\emptyset\neq S\subset S(\epsilon)$ for each $\epsilon>0$. We claim that $S$ is compact. Indeed, let $\{(x_n,y_n)\}$ be a sequence in $S$ and take $p_n\equiv p^*$. Then $\{(x_n,y_n)\}$ is a generalized approximating sequence corresponding to $\{p_n\}$ for SEP$(f,g,C,Q)$. By the definition of the generalized LP well-posedness by perturbations of SEP$(f,g,C,Q)$, the sequence $\{(x_n,y_n)\}$ has a subsequence converging to some element of $S$. Thus $S$ is indeed compact, as claimed.

Next, we consider condition (\ref{Thm3:con1}). Suppose to the contrary that condition (\ref{Thm3:con1}) does not hold true.
Then there exist a number $\tau>0$, a positive sequence $\{\epsilon_n\}$ decreasing to $0$, and a sequence $\{(x_n,y_n)\}$ such that $(x_n,y_n) \in S(\epsilon_n)$  for each natural number $n$ and
\begin{align}\label{Thm3:con2}
(x_n,y_n)\notin S+\mathbb{B}(0,\tau)\quad\forall n\in\mathbb{N}.
\end{align}

Since $(x_n,y_n)\in S(\epsilon_n)$, there exists $p_n\in\mathbb{B}(p^*,\epsilon_n)$ such that
\begin{equation*}
\begin{cases}
d(x_n,C)\leq\epsilon_n, \quad d(y_n,Q)\leq\epsilon_n,\\
\|y_n-Ax_n\|\leq \epsilon_n,\\
\tilde{f}(p_n, x_n, x)\geq -\epsilon_n\quad\forall x\in C,\\
\tilde{g}(p_n, y_n, y)\geq -\epsilon_n\quad\forall y\in Q.
\end{cases}
\end{equation*}
When combined with the fact that $p_n\rightarrow p^*$, this implies that $\{(x_n,y_n)\}$ is a generalized approximating sequence corresponding to $\{p_n\}$ for SEP$(f,g,C,Q)$.

By the definition of generalized LP well-posedness by perturbations of SEP$(f,g,C,Q)$, the sequence $\{(x_n,y_n)\}$ has a subsequence which converges to some element of $S$,
a contradiction to (\ref{Thm3:con2}).
Hence,
\begin{align*}
H(S(\epsilon),S)\rightarrow 0\text{ as }\epsilon\rightarrow 0.
\end{align*}

For the sufficiency part, suppose that $S$ is nonempty and compact, and that condition (\ref{Thm3:con1}) holds. Let $\{p_n\}\subset M$ with $p_n\rightarrow p^*$ and let $\{(x_n,y_n)\}$ be a generalized approximating sequence corresponding to $\{p_n\}$ for SEP$(f,g,C,Q)$. Then there exists a positive sequence $\{\epsilon_n\}$ such that $\epsilon_n\rightarrow 0$ and $(x_n,y_n)\in S(\tilde{\epsilon}_n)$ with
\begin{align*}
\tilde{\epsilon}_n=\max\{\epsilon_n, \|p_n-p^*\|\}\rightarrow 0.
\end{align*}

Using condition (\ref{Thm3:con1}), we get
\begin{align*}
d((x_n,y_n), S)&\leq D(S(\tilde{\epsilon}_n),S)\\&
=\max\{D(S(\tilde{\epsilon}_n),S), D(S,S(\tilde{\epsilon}_n))\}\\&
=H(S(\tilde{\epsilon}_n),S)\rightarrow 0.
\end{align*}

Since $S$ is compact, the sequence $\{(x_n, y_n)\}$ has a subsequence which converges to some element of $S$.
Hence SEP$(f,g,C,Q)$ is indeed LP well-posed by perturbations in the generalized sense.
\end{proof}

The following example illustrates Theorem \ref{Thm:3}.

\begin{example}
Let $E_1=E_2=\mathbb{R}$, $C=[-2,2]=Q$, $M=[0,2]\subset\mathbb{R}$, $p^*=1$, $A=I$, the identity operator.

Define $\tilde{f}:M\times E_1\times E_1\rightarrow\mathbb{R}$ and $\tilde{g}:M\times E_2\times E_2\rightarrow\mathbb{R}$ by
\begin{align*}
&\tilde{f}(p,z,x)=(x^2-p)^2-(z^2-p)^2, \quad p\in M, x,z\in E_1~~\text{and}\\
&\tilde{g}(p,w,y)=(y^2-p)^2-(w^2-p)^2, \quad p\in M, w,y\in E_2,
\end{align*}
with
\begin{align*}
&\tilde{f}(1,z,x)=(x^2-1)^2-(z^2-1)^2=f(z,x), \quad x,z\in E_1~~\text{and}\\
&\tilde{g}(1,w,y)=(y^2-1)^2-(w^2-1)^2=g(w,y), \quad w,y\in E_2.
\end{align*}
It is not difficult to check that $S=\{(-1,-1), (1,1)\}$ is the solution set of the $SEP(f,g,C,Q)$. Therefore, $S$ is nonempty and compact.
We have, for any $\epsilon>0$,
\begin{multline}
S(\epsilon)= \bigcup_{p\in[1-\epsilon,1+\epsilon]} \Big\{\left(z, w)\in\mathbb{R}\times\mathbb{R}: d(z,[-2,2])\leq\epsilon,  d(w,[-2,2])\leq\epsilon, \right.\nonumber\\
 \left. \vert w-z\vert\leq \epsilon; (x^2-p)^2-(z^2-p)^2\geq -\epsilon\quad\forall x\in [-2,2];\right.\nonumber\\ \left.(y^2-p)^2-(w^2-p)^2\geq -\epsilon\quad\forall y\in [-2,2]\right\}.
\end{multline}

Now,
\begin{align*}
&~~(x^2-p)^2-(z^2-p)^2\geq -\epsilon\quad\forall x\in [-2,2].\\&
\Leftrightarrow (z^2-p)^2\leq (x^2-p)^2+\epsilon\quad\forall x\in [-2,2].\\&
\Leftrightarrow z\in\left[-\sqrt{p+\sqrt{\epsilon}},-\sqrt{p-\sqrt{\epsilon}}\right]\bigcup\left[\sqrt{p-\sqrt{\epsilon}},\sqrt{p+\sqrt{\epsilon}}\right]
\end{align*}
for sufficiently small $\epsilon>0$ and $p\in[1-\epsilon,1+\epsilon]$.

Similarly,
\begin{align*}
&~~(y^2-p)^2-(w^2-p)^2\geq -\epsilon\quad\forall y\in [-2,2].\\&
\Leftrightarrow (w^2-p)^2\leq (y^2-p)^2+\epsilon\quad\forall y\in [-2,2].\\&
\Leftrightarrow w\in\left[-\sqrt{p+\sqrt{\epsilon}},-\sqrt{p-\sqrt{\epsilon}}\right]\bigcup\left[\sqrt{p-\sqrt{\epsilon}},\sqrt{p+\sqrt{\epsilon}}\right]
\end{align*}
for sufficiently small $\epsilon>0$ and $p\in[1-\epsilon,1+\epsilon]$.

Also,
\begin{align*}
\vert z-w\vert\leq\epsilon,\quad d(z,[-2,2])\leq\epsilon,\quad d(w,[-2,2])\leq\epsilon.
\end{align*}

Let
\begin{align*}
I(p,\epsilon)=\left[-\sqrt{p+\sqrt{\epsilon}},-\sqrt{p-\sqrt{\epsilon}}\right]\bigcup\left[\sqrt{p-\sqrt{\epsilon}},\sqrt{p+\sqrt{\epsilon}}\right]
\end{align*}
for sufficiently small $\epsilon>0$ and $p\in[1-\epsilon,1+\epsilon]$.

Combining all of the above, we see that
\begin{align*}
S(\epsilon)= \bigcup_{p\in[1-\epsilon,1+\epsilon]} \Big\{(z,w)\in I(p,\epsilon)\times I(p,\epsilon): \vert z-w\vert\leq\epsilon, d(z,[-2,2])\leq\epsilon, d(w,[-2,2])\leq\epsilon\Big\}
\end{align*}
for sufficiently small $\epsilon>0$.

Let $I_1(p,\epsilon)=\left[-\sqrt{p+\sqrt{\epsilon}},-\sqrt{p-\sqrt{\epsilon}}\right]$ and $I_2(p,\epsilon)=\left[\sqrt{p-\sqrt{\epsilon}},\sqrt{p+\sqrt{\epsilon}}\right]$ for sufficiently small $\epsilon>0$ and $p\in[1-\epsilon,1+\epsilon]$.

Then,  $S(\epsilon)\subset  \bigcup_{p\in[1-\epsilon,1+\epsilon]}\{(I_1(p,\epsilon)\times I_1(p,\epsilon))\cup (I_2(p,\epsilon)\times I_2(p,\epsilon))\}$ for sufficiently small $\epsilon>0$.

Also, we observe that for sufficiently small $\epsilon>0$ (see Fig. \ref{figaps}),
\begin{align*}
S(\epsilon)&\subset  \bigcup_{p\in[1-\epsilon,1+\epsilon]}\{(I_1(p,\epsilon)\times I_1(p,\epsilon))\cup (I_2(p,\epsilon)\times I_2(p,\epsilon))\}\\&
\subset  (I_1(\epsilon)\times I_1(\epsilon))\cup (I_2(\epsilon)\times I_2(\epsilon))=A(\epsilon) \text{ (say)},
\end{align*}
where 
\begin{align*}
&I_1(\epsilon)=\Big[-\sqrt{1+\epsilon+\sqrt{\epsilon}},-\sqrt{1-\epsilon-\sqrt{\epsilon}}\Big],\\&
I_2(\epsilon)=\Big[\sqrt{1-\epsilon-\sqrt{\epsilon}},\sqrt{1+\epsilon+\sqrt{\epsilon}}\Big].
\end{align*}

\begin{figure}[ht!]
\centering
\includegraphics[width=90mm]{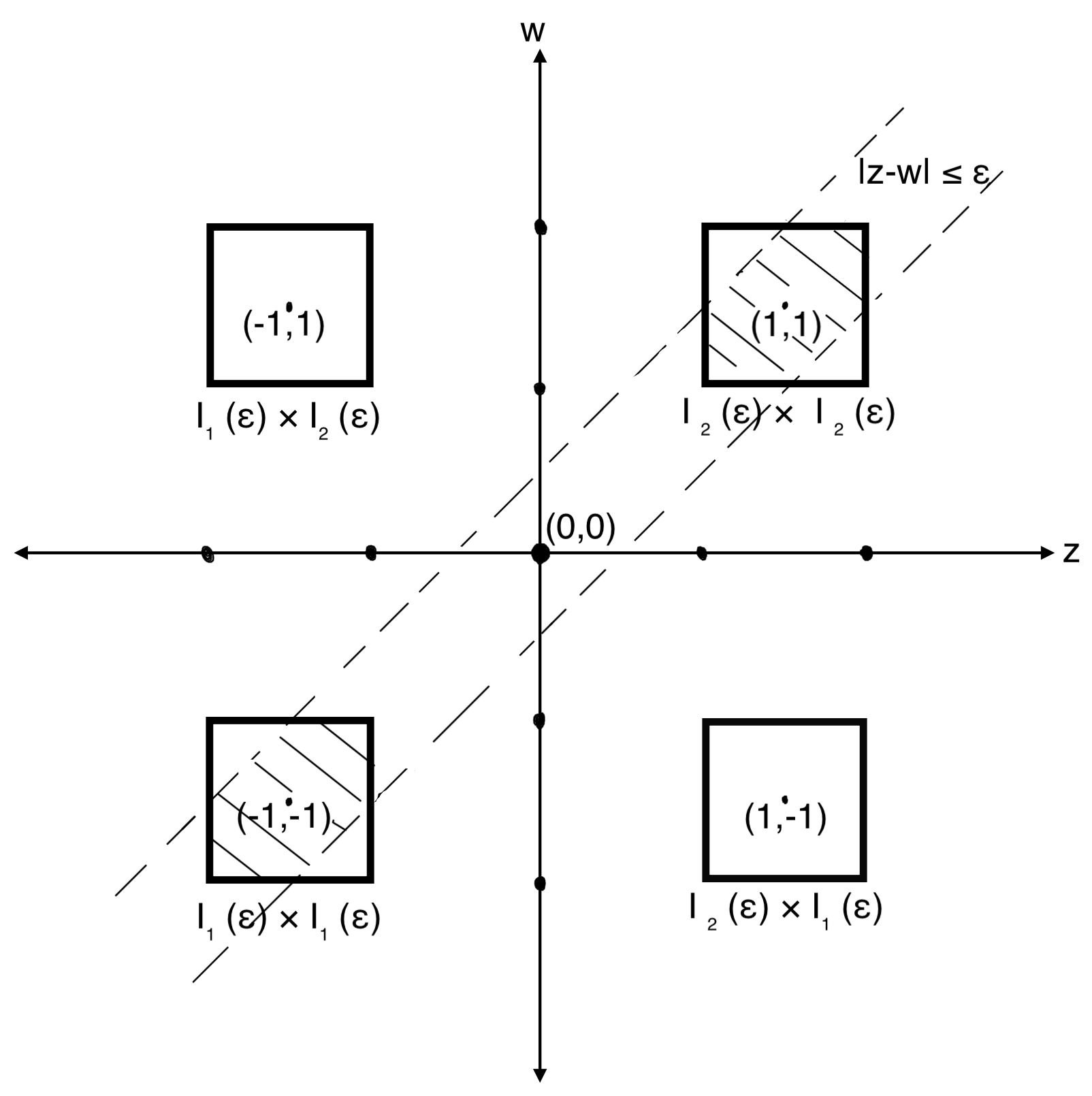}
\caption{ Note that the region given by $\vert z-w\vert\leq\epsilon$ may extend beyond $A(\epsilon)$. However, since $\epsilon>0$ is sufficiently small, $(I_1(\epsilon)\times I_2(\epsilon))\cup (I_2(\epsilon)\times I_1(\epsilon))$ will be outside the region. \label{overflow}}
\label{figaps}
\end{figure}

Now it is easy to check that $I_1(\epsilon)\times I_1(\epsilon)\rightarrow \{(-1,-1)\}$ and $I_2(\epsilon)\times I_2(\epsilon)\rightarrow \{(1,1)\}$ as $\epsilon\rightarrow 0$ with respect to the Hausdorff metric. Also, it is clear that $\cap_{\epsilon>0}(I_1(\epsilon)\times I_1(\epsilon))=\{(-1,-1)\}$ and $\cap_{\epsilon>0}(I_2(\epsilon)\times I_2(\epsilon))=\{(1,1)\}$ (see Fig. \ref{figaps}).

We know that $S(\epsilon)\supset S=\{(-1,-1), (1,1)\}$ for any $\epsilon>0$.

For sufficiently small $\epsilon>0$, the following inclusions hold:
\begin{align*}
S\subset S(\epsilon)\subset A(\epsilon).
\end{align*}

Let $x(\epsilon)\in A(\epsilon)$. It is not difficult to infer from the above observations that $d(x(\epsilon),S)\rightarrow 0$ as $\epsilon\rightarrow 0$.

Therefore,
\begin{align*}
H(A(\epsilon),S)=D(A(\epsilon),S)=\sup_{x(\epsilon)\in A(\epsilon)}d(x(\epsilon),S)\rightarrow 0\text{ as }\epsilon\rightarrow 0.
\end{align*}

This implies that $H(S(\epsilon),S)\leq H(A(\epsilon),S)\rightarrow 0\text{ as }\epsilon\rightarrow 0$.

Therefore, by the above Theorem \ref{Thm:3}, $SEP(f,g,C,Q)$ is LP well-posed by perturbations in the generalized sense.
\end{example}

We now proceed to establish the following theorem by assuming upper semi-continuity of $\tilde{f}$ and $\tilde{g}$, and by weakening the condition imposed on the solution set $S$ via the Kuratowski measure of noncompactness.

\begin{theorem}\label{Thm:4}
Let $N$ be finite-dimensional and let $\tilde{f}:M\times E_1\times E_1\rightarrow\mathbb{R}$ and $\tilde{g}:M\times E_2\times E_2\rightarrow\mathbb{R}$ be two functions such that $\tilde{f}(\cdot,\cdot,x)$ and $\tilde{g}(\cdot,\cdot,y)$ are upper semi-continuous for each $(x,y)\in E_1\times E_2$. Then SEP$(f,g,C,Q)$ is LP well-posed by perturbations in the generalized sense if and only if
\begin{align}\label{Thm4:con1}
S(\epsilon)\neq\emptyset\quad\forall\epsilon>0\quad\text{and}\quad \mu(S(\epsilon))\rightarrow 0 \text{ as } \epsilon\rightarrow 0.
\end{align}
\end{theorem}

\begin{proof}
Suppose SEP$(f,g,C,Q)$ is LP well-posed by perturbations in the generalized sense. Then $S(\epsilon)\neq\emptyset$ for all $\epsilon>0$. Appealing to Theorem \ref{Thm:3}, we get
\begin{align}\label{Thm4:con2}
H(S(\epsilon),S)\rightarrow 0\text{ as }\epsilon\rightarrow 0.
\end{align}
Under our assumptions, it is not difficult to show (using similar techniques to those employed in the proof of Theorem \ref{Thm:3}) that $S$ is compact.
Hence
$\mu(S)=0$.
Therefore we have
\begin{align}\label{Thm4:con3}
\mu(S(\epsilon))\leq 2H(S(\epsilon),S)+\mu(S)=2H(S(\epsilon),S).
\end{align}
Using (\ref{Thm4:con2}) and (\ref{Thm4:con3}), we see that
\begin{align*}
\mu(S(\epsilon))\rightarrow 0\text{ as } \epsilon\rightarrow 0.
\end{align*}

Conversely, suppose condition (\ref{Thm4:con1}) holds. We first show that the set $S(\epsilon)$ is closed for each $\epsilon>0$. Indeed, let $\epsilon$ be given and let $\{(x_n,y_n)\} \subset S(\epsilon)$ be such that $(x_n,y_n)\rightarrow(x^*,y^*)$. Our claim is that $(x^*, y^*) \in S(\epsilon)$.

Since $\{(x_n,y_n)\} \subset S(\epsilon)$, there exists $p_n\in\mathbb{B}(p^*,\epsilon)$ such that
\begin{equation*}
\begin{cases}
d(x_n,C)\leq\epsilon, \quad d(y_n,Q)\leq\epsilon,\\
\|y_n-Ax_n\|\leq \epsilon,\\
\tilde{f}(p_n, x_n, x)\geq -\epsilon\quad\forall x\in C,\\
\tilde{g}(p_n, y_n, y)\geq -\epsilon\quad\forall y\in Q.
\end{cases}
\end{equation*}

Without any loss of generality, we may assume that $p_n\rightarrow \bar{p}\in\mathbb{B}(p^*,\epsilon)$ because $N$ is finite-dimensional. Using the upper semi-continuity of $\tilde{f}(\cdot,\cdot,x)$ and $\tilde{g}(\cdot,\cdot,y)$ for each $(x,y)\in E_1\times E_2$ and Lemma \ref{Lem:2}, we get
\begin{equation*}
\begin{cases}
d(x^*,C)\leq\epsilon, \quad d(y^*,Q)\leq\epsilon,\\
\|y^*-Ax^*\|\leq \epsilon,\\
\tilde{f}(\bar{p}, x^*, x)\geq -\epsilon\quad\forall x\in C,\\
\tilde{g}(\bar{p}, y^*, y)\geq -\epsilon\quad\forall y\in Q.
\end{cases}
\end{equation*}
Therefore $S(\epsilon)$ is closed for each $\epsilon>0$.

Secondly, we claim that
\begin{align*}
S=\bigcap_{\epsilon>0} S(\epsilon).
\end{align*}

We already know that $S\subset S(\epsilon)$ for each $\epsilon>0$. Therefore, $S\subset\bigcap_{\epsilon>0} S(\epsilon)$.

Now we have to show the other inclusion. To this end, let $(x^*, y^*)\in\bigcap_{\epsilon>0} S(\epsilon)$.

Let $\{\epsilon_n\}$ be a positive sequence such that $\epsilon_n\rightarrow 0$ as $n\rightarrow\infty$ and $(x^*, y^*) \in S(\epsilon_n)$ for each $n\in\mathbb{N}$.
Then there exists $p_n\in\mathbb{B}(p^*,\epsilon_n)$ such that
\begin{equation*}
\begin{cases}
d(x^*,C)\leq\epsilon_n, \quad d(y^*,Q)\leq\epsilon_n,\\
\|y^*-Ax^*\|\leq \epsilon_n,\\
\tilde{f}(p_n, x^*, x)\geq -\epsilon_n\quad\forall x\in C,\\
\tilde{g}(p_n, y^*, y)\geq -\epsilon_n\quad\forall y\in Q.
\end{cases}
\end{equation*}

Under our assumptions, it is not difficult to show that
\begin{equation*}
\begin{cases}
x^*\in C, \quad y^*\in Q,\quad y^*=Ax^*,\\
f(x^*, x)\geq 0\quad\forall x\in C,\\
g(y^*, y)\geq 0\quad\forall y\in Q.
\end{cases}
\end{equation*}

Therefore, $(x^*,y^*)\in S$ and hence $\bigcap_{\epsilon>0} S(\epsilon)\subset S$.

Thus we have verified that
\begin{align*}
S=\bigcap_{\epsilon>0} S(\epsilon).
\end{align*}
Since $\mu(S(\epsilon))\rightarrow 0 \text{ as } \epsilon\rightarrow 0$, using \cite{KKUR1968} (page 412), one can show that $S$ is nonempty compact and
\begin{align*}
H(S(\epsilon),S)\rightarrow 0\text{ as }\epsilon\rightarrow 0.
\end{align*}
Therefore, in view of Theorem \ref{Thm:3}, SEP$(f,g,C,Q)$ is LP well-posed by perturbations in the generalized sense.
This completes the proof.
\end{proof}

\begin{remark}
The above theorem \ref{Thm:4} shows that the LP well-posedness by perturbations in the generalized sense of problem $(\ref{SEP})$ is related to the compactness of the approximate solution set.
\end{remark}

\section{Well-posedness and uniqueness of solutions}
\label{Sec:5}
One of the most interesting problems in the theory of well-posedness of variational inequalities is to establish the equivalence between well-posedness by perturbations and uniqueness of the solution. In this section we prove that the LP well-posedness by perturbations of the SEP$(f,g,C,Q)$ is equivalent to the existence and uniqueness of its solution under mild assumptions on the associated mappings.

\begin{theorem}\label{Thm:5}
Let $C$ and $Q$ be nonempty, closed and convex subsets of finite-dimensional real Banach spaces $E_1$ and $E_2$, respectively. Let $N$ be a normed space of parameters and let $M\subset N$ be a closed ball centered at $p^*\in M$.
Let $\tilde{f}:M\times E_1\times E_1\rightarrow\mathbb{R}$ and $\tilde{g}:M\times E_2\times E_2\rightarrow\mathbb{R}$ be functions which satisfy the following conditions:
\begin{itemize}
\item[(i)] $\tilde{f}(p^*,\cdot,\cdot)$ and $\tilde{g}(p^*,\cdot,\cdot)$ are hemicontinuous.
\item[(ii)]  $\tilde{f}(p,\cdot,\cdot)$ and $\tilde{g}(p,\cdot,\cdot)$ are both monotone for all $p\in M$.
\item[(iii)] $\tilde{f}(p,x,\cdot)$ and $\tilde{g}(p,y,\cdot)$ are both convex for all $p\in M$, $(x, y)\in E_1\times E_2$.
\item[(iv)] $\tilde{f}(p^*,x, x)$ and $\tilde{g}(p^*, y, y)$ are nonnegative for every $(x, y)\in C\times Q$.
\item[(v)] $\tilde{f}(\cdot,x,\cdot)$ and $\tilde{g}(\cdot,y,\cdot)$ are both continuous for all $(x, y)\in E_1\times E_2$.
\end{itemize}
Then SEP$(f,g,C,Q)$ is Levitin-Polyak (LP) well-posed by perturbations if and only if SEP$(f,g,C,Q)$ has a unique solution.
\end{theorem}

\begin{proof}
The necessity part of the theorem is straightforward and thus we present only the sufficiency part. Suppose that SEP$(f,g,C,Q)$ has a unique solution
and let $(x^*,y^*) \in C\times Q$ be the unique solution of SEP$(f,g,C,Q)$. That is, $(x^*,y^*) \in E_1\times E_2$ and
\begin{equation}\label{Thm5:inq1}
\begin{cases}
x^*\in C,\quad y^*\in Q,\quad y^*=Ax^*,\\
f(x^*,x)\geq 0\quad\forall x\in C,\\
g(y^*,y)\geq 0\quad\forall y\in Q.\\
\end{cases}
\end{equation}

Let $\{p_n\}\subset M$ with $p_n\rightarrow p^*$ and let $\{(x_n,y_n)\}$ be a generalized approximating sequence corresponding to $\{p_n\}$ for SEP$(f,g,C,Q)$. Then there exists a positive sequence $\{\epsilon_n\}$ such that $\epsilon_n \rightarrow 0$ and
\begin{equation}\label{Thm5:inq2}
\begin{cases}
d(x_n,C)\leq\epsilon_n, \quad d(y_n,Q)\leq\epsilon_n,\\
\|y_n-Ax_n\|\leq \epsilon_n,\\
\tilde{f}(p_n, x_n, x)\geq -\epsilon_n\quad\forall x\in C,\\
\tilde{g}(p_n, y_n, y)\geq -\epsilon_n\quad\forall y\in Q.
\end{cases}
\end{equation}

Since $\tilde{f}(p_n,\cdot,\cdot)$ and $\tilde{g}(p_n,\cdot,\cdot)$ are monotone, it follows from (\ref{Thm5:inq2}) that
\begin{equation}\label{Thm5:inq3}
\begin{cases}
\tilde{f}(p_n, x, x_n)\leq-\tilde{f}(p_n, x_n, x)\leq \epsilon_n\quad\forall x\in C,\\
\tilde{g}(p_n, y, y_n)\leq-\tilde{f}(p_n, y_n, y)\leq \epsilon_n\quad\forall y\in Q.
\end{cases}
\end{equation}


Since $\tilde{f}(\cdot,x,\cdot)$ and $\tilde{g}(\cdot,y,\cdot)$ are continuous, it follows from (\ref{Thm5:inq3}) that
\begin{equation}\label{Thm5:inq31}
\begin{cases}
\tilde{f}(p^*, x, x^*)=\lim_{n\rightarrow\infty}\tilde{f}(p_n, x, x_n)\leq \lim_{n\rightarrow\infty}\epsilon_n= 0\quad\forall x\in C,\\
\tilde{g}(p^*, y, y^*)=\lim_{n\rightarrow\infty}\tilde{g}(p_n, y,y_n)\leq\lim_{n\rightarrow\infty}\epsilon_n= 0\quad\forall y\in Q.
\end{cases}
\end{equation}

Using (\ref{Thm5:inq2}) and (\ref{Thm5:inq3}), we obtain
\begin{equation}\label{Thm5:inq4}
\begin{cases}
d(x_n,C)\leq\epsilon_n, \quad d(y_n,Q)\leq\epsilon_n,\\
\|y_n-Ax_n\|\leq \epsilon_n,\\
\tilde{f}(p_n, x, x_n)\leq\epsilon_n\quad\forall x\in C,\\
\tilde{g}(p_n, y, y_n)\leq\epsilon_n\quad\forall y\in Q.
\end{cases}
\end{equation}

Now we have to show that the sequence $\left\lbrace(x_n, y_n)\right\rbrace$ is bounded. Suppose to the contrary that the sequence $\left\lbrace(x_n, y_n)\right\rbrace$ is not bounded. With no loss of generality, we may assume that $\|(x_n, y_n)\|\rightarrow \infty.$ Set $u^*=(x^*, y^*)$, $u_n=(x_n, y_n)$ and
\begin{align*}
t_n&=\frac{1}{\|(x_n, y_n)-(x^*, y^*)\|}=\frac{1}{\|u_n-u^*\|},\\
v_n&=(z_n, w_n)=u^*+t_n(u_n-u^*)=((1-t_n)x^*+t_n x_n, (1-t_n)y^*+t_n y_n).
\end{align*}

We may assume that $t_n\in(0, 1]$ and $v_n\rightarrow\tilde{v}=(\tilde{z}, \tilde{w})\neq u^*$ because $E_1$ and $E_2$ are finite-dimensional.

Since $\tilde{f}$ is continuous in the first and third variables and convex in the third variable, we have
\begin{align}\label{Thm5:inq5}
f(x,\tilde{z})&=\tilde{f}(p^*,x,\tilde{z})\\&=\lim_{n\rightarrow\infty}\tilde{f}(p_n,x,z_n)\nonumber\\&
=\lim_{n\rightarrow\infty}\tilde{f}(p_n,x,(1-t_n)x^*+t_n x_n)\nonumber\\&
\leq\lim_{n\rightarrow\infty}[(1-t_n)\tilde{f}(p_n,x,x^*)+t_n\tilde{f}(p_n,x,x_n)]\nonumber\\&
\leq\lim_{n\rightarrow\infty} t_n\epsilon_n=0\quad\forall x\in C.\nonumber
\end{align}

Since $\tilde{g}$ is continuous in the first and third variables and convex in the third variable, we have
\begin{align}\label{Thm5:inq5}
g(y,\tilde{w})&=\tilde{g}(p^*,y,\tilde{w})\\&=\lim_{n\rightarrow\infty}\tilde{g}(p_n,y,w_n)\nonumber\\&
=\lim_{n\rightarrow\infty}\tilde{g}(p_n,y,(1-t_n)y^*+t_n y_n)\nonumber\\&
\leq\lim_{n\rightarrow\infty}[(1-t_n)\tilde{g}(p_n,y,y^*)+t_n\tilde{g}(p_n,y,y_n)]\nonumber\\&
\leq\lim_{n\rightarrow\infty} t_n\epsilon_n=0\quad\forall y\in Q.\nonumber
\end{align}

We also have
\begin{align}
\|A(z_n)-w_n\|&=\|A((1-t_n)x^*+t_n x_n)-(1-t_n)y^*-t_n y_n\|\\
&=t_n\|Ax_n-y_n\|\leq t_n\epsilon_n\rightarrow 0\text{ as } n\rightarrow\infty\nonumber.
\end{align}

Since $z_n\rightarrow\tilde{z}$ and $w_n\rightarrow\tilde{w}$, it follows that $A(\tilde{z})=\tilde{w}$.

Using Lemma \ref{Lem:2}, (\ref{Thm5:inq4}), and $x^*\in C$ $(d(x^*, C)=0)$, we get
\begin{align}\label{Thm5:inq6}
d(z_n,C)&=d((1-t_n)x^*+t_n x_n),C)\\&
\leq (1-t_n)d(x^*,C)+t_nd(x_n,C)\nonumber\\&
\leq t_n\epsilon_n\nonumber.
\end{align}
Letting $n\rightarrow\infty$ in (\ref{Thm5:inq6}), we find that
$\tilde{z}\in C$.
Similarly, we can show that $\tilde{w}\in Q$.
Using Lemma \ref{Lem:1} and the uniqueness of the solution to SEP, it is not difficult to show that $(\tilde{z}, \tilde{w})=(x^*, y^*)=u^*$, which is a contradiction.
Thus $\{u_n\}$ is bounded.

Without loss of generality, let $\left\lbrace u_{n^\prime}\right\rbrace$ be any convergent subsequence of $\left\lbrace u_n\right\rbrace$ with limit $\hat{u}=(\hat{x}, \hat{y})$. By our assumptions, it is not difficult to show that $\hat{u}=(\hat{x}, \hat{y})$ is a solution to SEP$(f,g,C,Q)$. Since $(x^*, y^*)=u^*$ is the unique solution to SEP$(f,g,C,Q)$, we find that  $\hat{u}=u^*$ and thus the generalized approximating sequence $\left\lbrace(x_n, y_n)\right\rbrace$ corresponding to $\{p_n\}$ converges to the unique solution to SEP$(f,g,C,Q)$. Therefore, the SEP$(f,g,C,Q)$ is indeed LP well-posed by perturbations.
This completes the proof.
\end{proof}

\section{Conclusion}
\label{Sec:6}

In this paper we introduced the perturbed split equilibrium problem and extended the notions of various kinds of Levitin-Polyak well-posedness by perturbations to the SEP$(f,g,C,Q)$ in infinite-dimensional real Banach spaces. We also showed that Levitin-Polyak well-posedness by perturbations of the SEP$(f,g,C,Q)$ is equivalent to the existence and uniqueness of its solution in finite-dimensional real Banach spaces. In addition, we provided some nontrivial examples to illustrate our theoretical results.

An open and interesting question arising from the above results and, in particular, from Theorem \ref{Thm:5}, is the validity of Theorem \ref{Thm:5} in infinite-dimensional real Banach spaces.

\section*{Acknowledgements}

The first author gratefully acknowledges the financial
support of the Post-Doctoral Program at the Technion -- Israel Institute of
Technology.
Simeon Reich was partially supported by the Israel Science Foundation (Grant 820/17), by
the Fund for the Promotion of Research at the Technion and by the Technion General Research Fund.


\section*{Declaration of interests}
The authors declare that they have no known competing financial interests that could have appeared to influence the work reported in this paper

\section*{Disclosure statement}
No potential conflict of interest was reported by the author.

\section*{Data availability statement}
The author acknowledge that the data presented in this study must be deposited and made
publicly available in an acceptable repository, prior to publication.



\bibliographystyle{amsplain}

\end{document}